\documentclass[a4paper]{amsart}
\usepackage{hyperref, aliascnt}
\usepackage{multirow}

\theoremstyle{definition}

\theoremstyle{definition}
\newtheorem{definition}{Definition}[section]
\theoremstyle{definition}

\theoremstyle{plain}
\newtheorem{theorem}[definition]{Theorem}
\newtheorem{lemma}[definition]{Lemma}

\theoremstyle{remark}
\newtheorem{remark}{Remark}

\def\today{\number\day\space\ifcase\month\or   January\or February\or
	March\or April\or May\or June\or   July\or August\or September\or
	October\or November\or December\fi\   \number\year}

\begin{document}

\title{Zero product and zero Jordan product determined Munn algebras}
\date{\today}

\author{Bo Yu}
\address{Bo Yu, 
	 School of Mathematics, East China University of Science and Technology\\
	 Shanghai 200237, P. R. China\\	} \email{modeace@163.com}

\author{Kaijia Luo}
\address{Kaijia Luo, 
	School of Mathematics, East China University of Science and Technology\\
	Shanghai 200237, P. R. China\\	}
\email{kaijia\_luo@163.com}

\author{Jiankui Li}
\address{Jiankui Li, 
School of Mathematics, East China University of Science and Technology\\
Shanghai 200237, P. R. China\\}
\email{jkli@ecust.edu.cn}

\thanks{This research was partly supported by the National Natural Science Foundation of China (Grant No.
	11871021).
}

\date{\today}

\begin{abstract} Let  $\mathfrak{M}(\mathbb{D}, m, n, P)$ be  the ring of all $m \times n$ matrices over a division ring $\mathbb{D}$, with the product given by  $A \bullet B=A P B$, where $P$ is a fixed $n \times m$ matrix over $\mathbb{D}$. 
	When $2\leq m, n <\infty$ and $\operatorname{rank} P \geq 2$, we demonstrate that
	every element in  $\mathcal{A}=\mathfrak{M}(\mathbb{D}, m, n, P)$ is
	a sum of finite products of pairs of commutators. We also estimate the minimal number $N$ such that  $\mathcal{A}= \sum^N [\mathcal{A}, \mathcal{A}][\mathcal{A}, \mathcal{A}]$. 
	Furthermore, if $\operatorname{char}\mathbb{D}\neq 2$, we prove that $\mathfrak{M}(\mathbb{D}, m, n, P)$  is additively spanned by Jordan products of idempotents.
	For a field  $\mathbb{F}$  with $\operatorname{char}\mathbb{F}\neq 2, 3$, we show that the Munn algebra $\mathfrak{M}(\mathbb{F}, m, n, P)$ is   zero product determined and  zero Jordan product determined.
\end{abstract}

\maketitle

\section{Preliminaries}

Building on the idea of Rees \cite{MR2893} who presented
0-simple semigroups by matrices over groups,   Munn introduced \cite{MR66355}  Munn algebras  in 1955 for
the finite dimensional case.  In this article,  we  generalize some results in matrix algebras to that in Munn algebras.

Let $\mathbb{R}$ be a ring. Recall that an element $e \in \mathbb{R}$ is called an \textit{idempotent} if $e^2=e$. 
The \textit{commutator} of two elements $x, y \in \mathbb{R}$ is defined as $[x, y] := xy - yx$. 
In Section 2, we show some results in  matrix algebras that $\mathrm{M}_n (\mathbb{R})$ can be generated by idempotents or commutators. Then we investigate  the generalization of these results to some Munn algebras.

In 2009, Bre\v{s}ar,    Gra\v{s}i\v{c}, and  Ortega \cite{MR2490691} introduced the concept of zero product determined algebras.  
The property of zero product determinacy is widely used to explore   derivations, homomorphisms, and preserver problems. 
Bre\v{s}ar summarized the relevant research  of the past 15 years   in \cite{MR4311188}. 
In Section 3, we  generalize the results of the matrix algebra $\mathrm{M}_n (\mathbb{R})$ being zero product determined and zero Jordan product determined to some Munn algebras.

\begin{definition}
	Let $\mathbb{R}$ be a ring with identity, $I$ and $\Lambda$ be nonempty sets, and $P=(P_{mi})_{m \in \Lambda, i \in I}$ be a $\Lambda \times I$ matrix over $\mathbb{R}$. We denote by $\mathfrak{M}(\mathbb{R}, I, \Lambda, P)$ the set of all $I \times \Lambda$-matrices 
	over $\mathbb{R}$ that have a finite number of nonzero entries.

	In the set $\mathfrak{M}(\mathbb{R}, I, \Lambda, P)$,
	we define the addition as the usual point-wise addition and the multiplication $\bullet$ with $A \bullet B=A P B$,
	where the multiplication on the right-hand side is the usual multiplication of matrices. With
	these operations, $\mathfrak{M}(\mathbb{R} , I, \Lambda, P)$ is a ring. We call $\mathfrak{M}(\mathbb{R} , I, \Lambda, P)$ a \textit{Rees matrix ring over $\mathbb{R}$ with sandwich matrix $P$}. 
\end{definition}

\begin{definition}
	Let $\mathbb{R}$ be an algebra over a field $\mathbb{F}$, $I$ and $\Lambda$ be nonempty sets, $P=(P_{mi})_{m \in \Lambda, i \in I }$ be a $\Lambda \times I$ matrix over $R$.  In  $\mathfrak{M}(\mathbb{R}, I, \Lambda, P)$,
	we define modular multiplication as the usual point-wise modular multiplication. Then $\mathfrak{M}(\mathbb{R}, I, \Lambda, P)$ is an $\mathbb{F}$-algebra. We call $\mathfrak{M}(\mathbb{R} , I, \Lambda, P)$  a \textit{Munn algebra over $\mathbb{R}$ with sandwich matrix $P$}. 
\end{definition}

For any $g \in \mathbb{R}$,  $i \in I$ and $s \in \Lambda$, 
let  $(g, i, s)$ denote the matrix in $\mathfrak{M}(\mathbb{R} , I, \Lambda, P)$ with entry $g$ on $(i,s)$ and  0 otherwise. 
Thus we have
$$
(g, i, s)(h, j, t)=(g p_{s j} h, i, t) \quad \text { for any } g, h \in \mathbb{R}, i, j \in I, s, t \in \Lambda.
$$

When studying certain  algebraic or ring properties of $\mathfrak{M}(\mathbb{D}, m, n, P)$, by Theorem \ref{th4.2} and Lemma \ref{lemma1.4}, we can consider $P$ as $E_r^{n, m}$.

\begin{theorem}\cite[VII. Theorem 2.6]{MR600654}\label{th4.2}
	If $P$ is an $n \times m$ matrix over a division ring $\mathbb{D}$ with $\operatorname{rank} P= r>0$, then $P$ is equivalent to $E_r^{n, m}$, 
	where $E_r^{n, m}=\left(\begin{array}{ll}I_{r} & 0 \\ 0 & 0\end{array}\right),$ and $I_{r}$
	is the $r \times r$ identity matrix.
\end{theorem}

\begin{lemma}\label{lemma1.4}
	Let $\mathbb{D}$ be a divison ring. If $V$ is an  invertible $m\times m$ matrix over $\mathbb{D}$, $W$ is an  invertible $n\times n$ matrix over $\mathbb{D}$, then $\mathfrak{M}(\mathbb{D}, m, n, P)$ and $\mathfrak{M}(\mathbb{D}, m, n, VPW)$ are ring isomorphic. Moreover, if $\mathbb{D}$ is a field, then $\mathfrak{M}(\mathbb{D}, m, n, P)$ and $\mathfrak{M}(\mathbb{D}, m, n, VPW)$ are algebra isomorphic. 
\end{lemma}

\begin{proof}
	Define the map $\Phi : \mathfrak{M}(\mathbb{D}, m, n, P) \rightarrow \mathfrak{M}(\mathbb{D}, m, n, VPW) $ as $\Phi (A)=W^{-1}AV^{-1}$. It is straightforward to verify that  
	$\Phi$ is an isomorphism.
\end{proof}

\section{Munn algebras generated by idempotents or commutators}
\begin{definition}
	Let $\mathbb{R}$ be a ring. For any $x, y \in \mathbb{R}$, define the \textit{Jordan product} by  $x \circ y= xy+yx$. If $\mathbb{R}$ is equal to its additive subgroup generated by elements of the form $e \circ f$  where $e, f \in \mathbb{R}$ are
	idempotents, then we say that $\mathbb{R}$ is \textit{additively spanned by Jordan products of its idempotents}.
\end{definition}

In \cite{MR4584284},  Bre\v{s}ar and  Godoy proved that if $\mathbb{R}$ is a unital ring containing $\frac{1}{2}$, then $\mathrm{M}_n(\mathbb{R})$ (for $n \geq 2$) is additively spanned by Jordan products of its idempotents. 
We present the following generalization:

\begin{theorem}\label{th4.3}
	Let $\mathbb{D}$ be	a division ring,   $P$ be  an $n \times m$ matrix over $\mathbb{D}$ with $\operatorname {rank}P=r$, and $\mathcal{A}=\mathfrak{M}(\mathbb{D}, m, n, P)$, where  $2\leq m, n< \infty$.
	\begin{itemize}
		\item If   $r \geq 1$, then $\mathcal{A} $ is generated by idempotents as  a $\mathbb{D}$-algebra.
		
		\item If   $r \geq 2$, then $\mathcal{A} $ is  generated by idempotents as  a ring.
		
		\item If  $r \geq 2$   and $\operatorname{char}\mathbb{D}\neq 2$, then $\mathcal{A} $ is additively spanned by Jordan products of idempotents.
	\end{itemize}
\end{theorem}

\begin{proof}
	If $r \geq 1$, then it  is easy to check that $(1, s, s)$, $(1, s, s) + (h, s, t) $ and $(1, s, s) + (h, t, s)$ are idempotents for any $1 \leq s \leq r$, $t\neq s$ and $h\in \mathbb{D}$.
	
	Since $(h, s, t)=h(1, s, 1) (1, 1, t)$ holds for any $1 \leq s \leq m$, 	$1 \leq t \leq n$ and $h\in \mathbb{D}$, it follows that $\mathcal{A}$
	is  generated by idempotents as a $\mathbb{D}$-algebra.
	
	From now on, we assume that $r \geq 2$. Then for any $1 \leq s \leq r$ and $h\in \mathbb{D}$, there exists a $t \neq s$ with $1 \leq t \leq r$  such that
	$(h, s, s) = (h, s, t) (1, t, s) $. Furthermore, we have  $(h, s, t) = (h, s, 1) (1, 1, t)$ for any $1 \leq s \leq m$, $1 \leq t \leq n$ and $h \in \mathbb{D}$. It follows that $\mathcal{A}$
	is  generated by idempotents as a ring.

	It is clear that $(h, s, s) + (h, t, s) + (1 - h, s, t) + (1 - h, t, t)$ is an idempotent for any $1 \leq  s \neq  t  \leq r$ and $h \in \mathbb{D}$. 
	
	If $\operatorname{char}\mathbb{D}\neq 2$, then we can obtain that
	$$\begin{aligned}
		(h, s, s)
		=&((\frac{1}{2}h, s, s) + (\frac{1}{2}h, t, s)+(1 - \frac{1}{2}h, s, t) + (1 - \frac{1}{2}h, t, t)) \circ (1, s, s)\\
		&- ((1, s, s) + (\frac{1}{4}h, t, s)) \circ  ((1, s, s) + (\frac{1}{4}h, t, s)) + (1, s, s) \circ (1, s, s)\\
		&- ((1, s, s) + (\frac{1}{2} - \frac{1}{4}h, s, t)) \circ ((1, s, s) + (\frac{1}{2} - \frac{1}{4}h, s, t)) + (1, s, s) \circ (1, s, s)
	\end{aligned}$$	
	for any $1\leq s \neq  t \leq r$	and $h \in \mathbb{D}$. 
	Since $(h, s, t)=((1, 1, 1)+(h, s, 1))\circ(1, 1, t)-(1, 1, 1)\circ(1, 1, t)
	$ for any $r < s \leq m$, $r <t \leq n$ and $h \in \mathbb{D}$, 
	it can be deduced that $\mathcal{A}$ is additively spanned by Jordan products of idempotents.
\end{proof}

\begin{definition}
	Let $\mathbb{R}$ be a ring. 
	For subsets $X, Y \subseteq \mathbb{R}$,  define
	$
	[X, Y]:=\{[x, y]: x \in X, y \in Y\}$ 
	and $X  Y:=\{x y: x \in X, y \in Y\}. 
	$
	We also write $X^2$ for $X  X$.
\end{definition}

In \cite{arxiv1},  Gardella and  Thiel studied the rings and $C^*$-algebras generated by commutators. They proved that  a unital ring $\mathbb{R}$ is generated by its commutators as an ideal if and only if  there exists a natural number $N$ such that every element in $\mathbb{R}$
is a sum of $N$ products of pairs of commutators.

For a division ring $\mathbb{D}$, $\mathfrak{M}(\mathbb{D} , m, n, P)$ is a ring that might not contain an identity. We give the necessary and sufficient conditions for every element in  $\mathfrak{M}(\mathbb{D} , m, n, P)$
to be a sum of finite products of pairs of commutators. 
Before that, we introduce  Baxter's result   in \cite{MR180577}.
\begin{theorem}\cite{MR180577}\label{th4.6}
	Let $\mathbb{R}$ be a simple associative ring. Then either $\mathbb{R}$ is a
	field or 
	every  element in $\mathbb{R}$ is a sum of products of pairs of commutators, i.e., 
	$\mathbb{R}= \sum [\mathbb{\mathbb{R}}, \mathbb{\mathbb{R}}][\mathbb{\mathbb{R}}, \mathbb{\mathbb{R}}]$.
\end{theorem}

\begin{theorem}\label{th4.7}
	Let $\mathbb{D}$ be	a division ring,   $P$ be  an $n \times m$ matrix over $\mathbb{D}$ with $\operatorname {rank}P=r$, and $\mathcal{A}=\mathfrak{M}(\mathbb{D}, m, n, P)$, where  $2\leq m, n< \infty$.  
\begin{itemize}
		\item  If $r\geq 2$, then 
		$\mathcal{A}= \sum [\mathcal{A}, \mathcal{A}][\mathcal{A}, \mathcal{A}]$.
		
		\item 	If $r = 1$, then  $\mathcal{A}= \sum [\mathcal{A}, \mathcal{A}][\mathcal{A}, \mathcal{A}]$ if and only if $\mathbb{D}$ is not commutative.   
	\end{itemize}	
	
\end{theorem}

\begin{proof}
	\textbf{Case 1.} $r\geq 2$.
	
	Let $i, j$ be  arbitrary natural numbers with $1 \leq i \leq m$, $1 \leq j \leq n$ and $a, b, c \in \mathbb{D}$. 
	
	If $i \neq 1$ and $j\neq 1$, then
	$(a, i ,1)=
	[(a, i, 1), (1, 1, 1)]
	$
	and
	$(b, 1 ,j)=
	[(1, 1, 1), (b, 1, j)].
	$
	Thus 	$$(a, i, j)=(a, i, 1)(1, 1, j)\in [\mathcal{A}, \mathcal{A}][\mathcal{A}, \mathcal{A}]$$  for any $i \neq 1$, $j\neq 1$ and $a \in \mathbb{D}$. 
	Similarly, we have 
	$$(a, i, j)=(a, i, 2)(1, 2, j)\in [\mathcal{A}, \mathcal{A}][\mathcal{A}, \mathcal{A}]$$
	for any $i \neq 2$, $j\neq 2$ and $a \in \mathbb{D}$.  
	Thus for any $b, c \in \mathbb{D}$, we have
	$$
	(b, 1, 2)+(c, 2, 1)
	=[(1, 1, 2), (1, 2, 1)][(1, 1, 1), ((b, 1, 1)+(c, 2, 1))]\in [\mathcal{A}, \mathcal{A}][\mathcal{A}, \mathcal{A}].
	$$
	It follows that for any $A=\sum_{1 \leq i \leq m, 1 \leq j \leq n}(a_{ij}, i, j) \in \mathcal{A}$,  $A$ can be written as the sum of the above three elements, which means that $\mathcal{A}= \sum [\mathcal{A}, \mathcal{A}][\mathcal{A}, \mathcal{A}]$.
	
	\textbf{Case 2.1. } $r= 1$ and  $\mathbb{D}$ is not commutative.
	
	By Theorem \ref{th4.6}, for any $x\in \mathbb{D}$, there exist   $s\in \mathbb{N}$ and $\{a_k\}_{i=k}^{s}, \{b_k\}_{k=1}^{s}, \{c_k\}_{k=1}^{s}, \{d_k\}_{k=1}^{s} \subseteq  \mathbb{D}$  such that 
	$x=\sum_{k=1}^{s}[a_k, b_k][c_k, d_k]$. 
	Thus, we have 
	$$(x, 1, 1)=\sum_{k=1}^{s}[(a_k, 1, 1), (b_k, 1, 1)][(c_k, 1, 1), (d_k, 1, 1)]\in [\mathcal{A}, \mathcal{A}][\mathcal{A}, \mathcal{A}].$$
	
	Since $\mathbb{D}$ is not commutative, there exist $a, b \in \mathbb{D}$ such that $[a, b]$ is invertible. Thus for any $i \neq 1$ and $j\neq 1$, we have
	$$
	\begin{aligned}
		(x, 1, j)& = [(a, 1, 1), (b, 1, 1)] [(1, 1, 1), ([a, b]^{-1}x, 1, j)],\\
		(x, i, 1)& = [(x[a, b]^{-1}, i, 1), (1, 1, 1)] [(a, 1, 1), (b, 1, 1)],\\
		(x, i, j)& =    [(x, i, 1), (1, 1, 1)] [(1, 1, 1), (1, 1, j)].
	\end{aligned}
	$$	
	It follows that $\mathcal{A}= \sum [\mathcal{A}, \mathcal{A}][\mathcal{A}, \mathcal{A}]$.
	
	\textbf{Case2.2. } $r= 1$ and  $\mathbb{D}$ is a field.
	
	In this case, for any matrix $T \in [\mathcal{A}, \mathcal{A}]$, 
	the element at position $(1, 1)$ in $T$ is zero, which means that $\mathcal{A} \neq \sum[\mathcal{A}, \mathcal{A}][\mathcal{A}, \mathcal{A}]$.
\end{proof}

In \cite{MR1908928},  Pop proved that
elements in unital $C^*$-algebras without tracial states can be represented as finite sums of commutators and he   defined the invariant $\nu(\mathcal{A})$ as the least positive integer $N$ such that $\mathcal{A}=\sum^N[\mathcal{A}, \mathcal{A}]$. 
Inspired by  Pop's work,  Gardella and  Thiel \cite{arxiv1} introduced the invariant $\xi(\mathbb{R})$. 
\begin{definition}
	For a ring $\mathbb{R}$, if there exists an $N\in \mathbb{N}$ such that $\mathbb{R} = \sum^{N}[\mathbb{R}, \mathbb{R}][\mathbb{R}, \mathbb{R}]$, define
	$$
	\xi(\mathbb{R}):=\min \left\{N \in \mathbb{N}: \mathbb{R}=\sum\nolimits^N[\mathbb{R}, \mathbb{R}][\mathbb{R}, \mathbb{R}]\right\}.
	$$
\end{definition}

In \cite[Theorem 5.5]{arxiv1},  Gardella and  Thiel proved that
for a unital ring $\mathbb{R}$  and $n \geq 2$,  $\xi(\mathrm{M}_n(\mathbb{R})) \leq 2$. 
For $\mathcal{A}=\mathfrak{M}(\mathbb{D}, m, n, P)$, 
we estimate  $\xi(\mathcal{A})$ based on the rank of $P$. 

\begin{theorem}\label{th4.10}
	Let $\mathbb{D}$ be	a division ring,   $P$ be  an $n \times m$ matrix over $\mathbb{D}$ with $\operatorname {rank}P=r$, and $\mathcal{A}=\mathfrak{M}(\mathbb{D}, m, n, P)$, where  $2\leq m, n< \infty$. 	
	If $r = \operatorname{min} \{m, n\}$, then $\xi(\mathcal{A}) \leq 2$.
\end{theorem}

\begin{proof}
	Without loss of generality, we may assume  that $r = m \leq n$.
	
	\textbf{Case 1.} $m = 2$.
	
	For any given $a_{1, p}, 2 \leq p \leq n$ and $a_{2, 1}$ in $\mathbb{D}$, we have
	$$
	\left(\begin{matrix}
		0 & a_{1, 2} & \cdots & a_{1, n} \\
		a_{2, 1} & 0 &  \cdots & 0
	\end{matrix}\right)
	=
	\left[\left(\begin{matrix}
		1 & 0  & \cdots & 0\\
		0 & 0  & \cdots & 0
	\end{matrix}\right),\left(\begin{matrix}
		0 & a_{1, 2}  & \cdots & a_{1, n} \\
		-a_{2, 1} & 0  & \cdots & 0
	\end{matrix}\right)\right]
	\in \left[ \mathcal{A}, \mathcal{A} \right].
	$$	
	Further, we have
	\begin{equation}\label{eq4.2}
		\left(\begin{matrix}
			1 & 0 & 0 & \cdots & 0\\
			0 & -1 & 0 & \cdots & 0
		\end{matrix}\right) 
		=
		\left[\left(\begin{matrix}
			0 & 1 & 0 & \cdots & 0 \\
			0 & 0 & 0 & \cdots & 0
		\end{matrix}\right),\left(\begin{matrix}
			0 & 0 & 0 & \cdots & 0\\
			1 & 0 & 0 & \cdots & 0
		\end{matrix}\right)\right] \in \left[\mathcal{A}, \mathcal{A}\right] .
	\end{equation}
	Thus, for any $a_{i, j} \in \mathbb{D}, 1 \leq i \leq 2$ and $1 \leq j \leq n$, we have
	$$
	\begin{aligned}
		\left(\begin{matrix}
			a_{1, 1} & a_{1, 2} & a_{1, 3} & \cdots & a_{1, n}\\
			a_{2, 1} & a_{2, 2} & a_{2, 3} & \cdots & a_{2, n}
		\end{matrix}\right)
		=&\left(\begin{matrix}
			1 & 0 & 0 & \cdots & 0\\
			0 & -1 & 0 & \cdots & 0
		\end{matrix}\right)\left(\begin{matrix}
			0 & a_{1, 2} & a_{1, 3} & \cdots & a_{1, n} \\
			-a_{2, 1} & 0 & 0 & \cdots & 0
		\end{matrix}\right)\\
		+\left(\begin{matrix}
			0 & 1 & 0 & \cdots & 0\\
			1 & 0 & 0 & \cdots & 0
		\end{matrix}\right) & 
		\left(\begin{matrix}
			0 & a_{2, 2} & a_{2, 3} & \cdots & a_{2, n}\\
			a_{1, 1} & 0 & 0 & \cdots & 0
		\end{matrix}\right) 
		\in \sum\nolimits^2\left[\mathcal{A}, \mathcal{A}\right]^2.
	\end{aligned}
	$$
	
	\textbf{Case 2.} $m \geq 3$.
	
	Let $a \in \mathcal{A}$.
	As the proof in \cite[Theorem 5.5]{arxiv1}, we 
	set $$
	x := \sum\nolimits_{j=1}^{m-1}(1, j+1, j), \quad 
	y := \sum\nolimits_{j=1}^{m-1}(1, j, j+1), \quad 
	c := a + \sum\nolimits_{k=2}^{m}x^{k-1}ay^{k-1}.
	$$	
	Thus we have $yx = \sum_{j=1}^{m-1}(1, j, j)$, and $xcy = c - a$. Therefore $[y, xc] = a - (1, m, m)c$, 
	which implies that the entries of the matrix $[y, xc]$
	match those of $a$ in the first $( n - 1 )$ rows.

	Depending on the parity of $m$, we construct the following:
	$$d :=
	\begin{cases}
		\sum\nolimits_{j=1}^{(m-1)/2}((1, 2j-1, 2j-1) + (-1, 2j, 2j)), & $m$ \text{ is odd},\\
		\sum\nolimits_{j=1}^{(m-2)/2}((1, 2j-1, 2j-1) + (-1, 2j, 2j)), & $m$ \text{ is even}.
	\end{cases}
	$$
	
	Therefore by \eqref{eq4.2}, we have 
	$$d =
	\begin{cases}
		[\sum_{j=1}^{(m-1)/2}(1, 2j-1, 2j),  \sum_{j=1}^{(m-1)/2}(1, 2j, 2j-1)], & $m$ \text{ is odd},\\
		[\sum_{j=1}^{(m-2)/2}(1, 2j-1, 2j),  \sum_{j=1}^{(m-2)/2}(1, 2j, 2j-1)], & $m$ \text{ is even}.
	\end{cases}
	$$	
	
	It follows that   $da = d[y, xc]\in [\mathcal{A}, \mathcal{A}]^2$ for any $a \in \mathcal{A}$ where $da$  is an arbitrary matrix in  $\mathcal{A}$ 
	whose last $p_m$ rows are zero, and
	$p_m = 2-(m \operatorname{mod} 2)$. 
	Similarly, every matrix in $\mathcal{A}$ whose first $p_m$ rows are zero  can also be expressed as a product of two commutators. 
	Thus for any $m \geq 3$, every matrix in $\mathcal{A}$ belongs to $\sum^2\left[\mathcal{A}, \mathcal{A}\right]^2$.
\end{proof}

When $P$ is not of full rank,
we have the following estimate for $\xi(\mathcal{A})$.

\begin{theorem}\label{th4.11}
	Let $\mathbb{D}$ be	a division ring,   $P$ be  an $n \times m$ matrix over $\mathbb{D}$ with $\operatorname {rank}P=r \geq 2$, and $\mathcal{A}=\mathfrak{M}(\mathbb{D}, m, n, P)$, where  $2\leq m, n< \infty$. 
	Then $\lceil \min\{m, n\}/r\rceil \leq \xi(\mathcal{A})\leq \lceil \min\{m, n\}/r\rceil+3$, where $\lceil x\rceil$ denotes 
	the smallest integer $\geq x$. 
\end{theorem}

\begin{proof}
	Without loss of generality, we may assume  that $m \leq n$.
	
	For any $A=\sum\limits_{\substack{1 \leq i \leq r\\ 1 \leq j \leq n}}(x_{i, j}, i, j)$, 
	by Theorem \ref{th4.10}, we have $A\in \sum^{2}[\mathcal{A}, \mathcal{A}]^2$.\\	
	For any $B_k=\sum\limits_{\substack{kr+1 \leq i \leq  (k+1)r\\ r+1 \leq j \leq n}}(x_{i, j}, i, j)$,
	where $k \in \mathbb{N}_+$ satisfies $kr \leq m$, we have 
	$$\begin{aligned}
		B_k&=(\sum_{i=kr+1}^{(k+1)r}(1, i, i-kr)) \cdot( \sum_{\substack{1 \leq i \leq  r\\ r+1 \leq j \leq n}}(x_{i+kr, j}, i, j))\\
		&=[\sum_{i=kr+1}^{(k+1)r}(1, i, i-kr), \sum_{i=1}^{r}(1, i, i)] \cdot [\sum_{1=1}^{r}(1, i, i), \sum_{\substack{1 \leq i \leq  r\\ r+1 \leq j \leq n}}(x_{i+kr, j}, i, j)] \in [\mathcal{A}, \mathcal{A}]^2.
	\end{aligned}
	$$	
	For any $C=\sum\limits_{\substack{gr+1 \leq i \leq  m\\ r+1 \leq j \leq n}}(x_{i, j}, i, j)$,
	where $g \in \mathbb{N}_+$  satisfies $gr < m <(g+1)r$, we have
	$$\begin{aligned}
		C&=(\sum_{i=gr+1}^{m}(1, i, i-gr)) \cdot (\sum_{\substack{1 \leq i \leq  m-gr\\ r+1 \leq j \leq n}}(x_{i+gr, j}, i, j))\\
		&=[\sum_{i=gr+1}^{m}(1, i, i-gr), \sum_{i=1}^{r}(1, i, i)] \cdot [\sum_{1=1}^{m-gr}(1, i, i), \sum_{\substack{1 \leq i \leq  m-gr\\ r+1 \leq j \leq n}}(x_{i+gr, j}, i, j)] \in [\mathcal{A}, \mathcal{A}]^2.
	\end{aligned}$$	
	For any $D=\sum\limits_{\substack{r+1 \leq i \leq  m\\ 1 \leq j \leq r}}(x_{i, j}, i, j)$, by Theorem \ref{th4.10}, we have $D\in \sum^{2}[\mathcal{A}, \mathcal{A}]^2$. 
	
	Indeed, $A, B_k, C, D$ are as follows: 
	\begin{center}
		\begin{tabular}{|c|c|}
			\hline
			\multicolumn{2}{|c|}{$A$} \\
			\hline
			\multirow{3}{*}{$D$} & \makebox[10em]{$B_1$} \\
			\cline{2-2}
			& \makebox[10em]{$B_2$} \\
			\cline{2-2}
			& \makebox[10em]{$\vdots$} \\
			\cline{2-2}
			& \makebox[10em]{$C$} \\
			\hline
		\end{tabular}.
	\end{center}

	Therefore, for any $T \in \mathcal{A}$, there exist $A, B_k, C, D$ of the above form  such that 
	$$
	T=A+\sum_{k=1}^{\lceil m/r\rceil-2}B_k+C+D \in \sum\limits^{ \lceil m/r\rceil+3}[\mathcal{A}, \mathcal{A}]^2,
	$$
	which means	that $\xi(\mathcal{A})\leq \lceil \min\{m, n\}/r\rceil+3$.  It is evident that $\lceil \min\{m, n\}/r\rceil \leq \xi(\mathcal{A})$.
\end{proof}

In \cite{Arxiv2},  Bre\v{s}ar,  Gardella and  Thiel proved that for some special  matrix  rings,  every element in these matrix  rings is a product of two commutators. 
Based on their proof, we can derive the following theorem.

\begin{theorem}		Let $\mathbb{D}$ be	a division ring,   $P$ be  an $n \times m$ matrix over $\mathbb{D}$ with $\operatorname {rank}P=r \geq 2$, and $\mathcal{A}=\mathfrak{M}(\mathbb{D}, m, n, P)$, where  $2\leq m, n< \infty$. 
	If  one of the following statements holds:
	\begin{itemize}
		\item\label{i}  	$\operatorname{min} \{m, n\} =2$,
		
		\item\label{ii}
		$\mathbb{D}$ is a division ring with infinite center,
	\end{itemize}
	then $\xi(\mathcal{A}) =1$.
\end{theorem}

\begin{proof}
	Without loss of generality, we may assume  that $m \leq n$ and $P=\begin{pmatrix}
		I_m \\
		0 \\
	\end{pmatrix}
	$. 
	For any  $A \in \mathcal{A}$, denote
	$A = \begin{pmatrix}
		A_{11} & A_{12}\\
	\end{pmatrix}
	$. By the proof of  \cite[Theorem 4.4]{Arxiv2}, if (\ref{i}) or (\ref{ii}) holds, then for any $A_{11}$ in $\mathrm{M}_n(\mathbb{D})$, there exist $B_{11}, C_{11}, D_{11}, E_{11}  \in \mathrm{M}_n(\mathbb{D})$ such that $A_{11} = [B_{11}, C_{11}][D_{11}, E_{11}]$, and such that $ [B_{11}, C_{11}]$ and $D_{11}$ are
	invertible. Thus 
	$$\begin{aligned}
	\begin{pmatrix}
		A_{11} & A_{12}\\
	\end{pmatrix}
	&=[\begin{pmatrix}
		B_{11} & 0\\
	\end{pmatrix}, \begin{pmatrix}
		C_{11} & 0\\
	\end{pmatrix}][\begin{pmatrix}
		D_{11} & 0\\
	\end{pmatrix}, \begin{pmatrix}
		E_{11}  &~~ D_{11}^{-1}[B_{11}, C_{11}]^{-1}A_{12}\\
	\end{pmatrix}]\\
	&\in [\mathcal{A}, \mathcal{A}]^2,
	\end{aligned}
	$$
	which means that $\xi(\mathcal{A}) =1$. 
\end{proof}

\section{zpd  and zJpd Munn algebra}

\begin{definition}\label{defzpd}
	An algebra $\mathcal{A}$ over a field $\mathbb{F}$ is said to be \textit{zero product determined} if for every bilinear functional $\phi : \mathcal{A} \times \mathcal{A} \longrightarrow \mathbb{F}$ with the property that 
	\begin{equation}\label{zp}
		xy=0 \Longrightarrow \phi(x, y)=0 \quad \text{for all}\quad  x, y \in \mathcal{A},
	\end{equation}
	there exists a linear functional  $\tau$ on $\mathcal{A}$ such that 
	$$
	\phi(x, y)=\tau(xy) \quad \text{for all} \quad x, y \in \mathcal{A}.
	$$
	We  use \textit{zpd} as an abbreviation for zero product determined.
	
	If $\mathcal{A}$ is a Banach algebra and $\phi$ and $\tau$ are continuous, then we call $\mathcal{A}$ a \textit{zpd Banach algebra}. For more information, see \cite{MR4311188}.	
\end{definition}

\begin{lemma}\label{danwei} 
	Let $\mathbb{D}$ be a division ring with $\operatorname{char}\mathbb{D} \neq 2, 3$ and $C\in M_n(\mathbb{D})$. 
	Denote the space  consists of all column vectors $x=(x_1, x_2, \cdots, x_n)^T$ with $n$ components in $\mathbb{D}$ by $\mathbb{D}^n$.
	If for any $a, b \in \mathbb{D}^n$, $a^T b=0$ implies $a^T C b=0$, then $C=\lambda I$ for some $\lambda \in \mathbb{D}$.
\end{lemma}

\begin{proof}
	Denote $C=(C_{i, j})_{n \times n}$.
	Let $a=(1, 0, \cdots, 0,  y)^T, b=(1, 0, \cdots, 0, -y^{-1})^T$ for any $y \neq 0 \in \mathbb{D}$. Thus we have $a^T C b=C_{1, 1}-C_{1, n}y^{-1}+yC_{n, 1}-yC_{n, n}y^{-1}=0$. 
	
	Replacing $y$ by $1, -1, 2$ respectively, we can obtain that
	$$C_{1, 1}=C_{n, n}\quad \text{and} \quad C_{1, n} = C_{n, 1} = 0.$$

	Similarly, we can obtain $C_{i, i}=C_{j, j}$ and $C_{i, j}=0$ for any  $1 \leq i \neq j \leq n$ which means that $C=\lambda I$ for some $\lambda \in \mathbb{D}$.
\end{proof}

\begin{theorem}
	Let $\mathbb{F}$ be a field with $\operatorname{char}\mathbb{F} \neq 2, 3$. 
	If $2 \leq  m, n < \infty$, and $\operatorname{rank}P=r \geq 2$,  then $\mathcal{A}=\mathfrak{M}(\mathbb{F}, m, n, P)$ is  zpd.
\end{theorem}

\begin{proof}
	Take a bilinear functional $\phi : \mathcal{A} \times \mathcal{A} \longrightarrow \mathbb{F}$ satisfying \eqref{zp}.  
	For any $1\leq i, p \leq m$ and  $1 \leq j, q \leq n$, 
	we denote $\phi((1, i, j), (1, p, q))=\lambda^{iq}_{jp}$.
	
	If $r< j \leq n$ or $r< p \leq m$, then $(1, i, j) (1, p, q)=0$. Thus by the assumption of $\phi$, we have $\lambda^{iq}_{jp}=0$ when $r< j \leq n$ or $r< p \leq m$.
	
	For any given $1 \leq i \leq m$ and  $1\leq q \leq n$, let 
	$$C=\sum_{1 \leq j \leq n}(c_{ij}, i, j)\quad \text{and}\quad D=\sum_{1 \leq q\leq m}(d_{pq}, p, q),
	$$
	where $c_{ij}$ and $d_{pq}$ are arbitrary elements in $\mathbb{F}$.
	
	Thus we have
	$$	CD=(\sum_{1 \leq k \leq r}c_{ik}d_{kq}, i, q)
	$$
	and 
	$$		\phi(C, D)= \sum_{\substack{ 1 \leq j \leq n\\ 1 \leq p \leq m}}\phi((c_{ij}, i, j), (d_{pq}, p, q))
	=\sum_{\substack{ 1 \leq j \leq r\\ 1 \leq p \leq r}}c_{ij}d_{pq}\lambda^{iq}_{jp}.
	$$	
	It follows from \eqref{zp} that $\sum_{1 \leq k \leq r} c_{ik} d_{kq}=0$ implies $\sum_{ 1 \leq j \leq r,  1 \leq p \leq r} c_{ij} d_{pq} \lambda^{iq}_{jp}=0$.
	
	Regard $(\lambda_{jp}^{iq})_{1 \leq j, p \leq r}$ as an $r \times r$ matrix. 
	Thus by Lemma \ref{danwei}, we have $(\lambda_{jp}^{iq})_{1 \leq j, p \leq r}=\lambda^{iq} I$ for some  $\lambda^{iq} \in \mathbb{F}$, which means that $\lambda_{jj}^{iq}
	=\lambda^{iq}$  and $\lambda_{jp}^{iq}=0$ for any $1 \leq j \neq p \leq r$.  
	
	Then we can define a  linear functional $\tau$ on $\mathcal{A}$ by
	$$\tau(\sum_{\substack{ 1 \leq i \leq m \\ 1 \leq q \leq n}}(a_{iq}, i, q)):=\sum_{\substack{r+1 \leq i \leq m \\ r+1 \leq q \leq n}}a_{iq}\lambda^{iq}.
	$$
	Hence for any $A=\sum\limits_{\substack{ 1 \leq i \leq m \\ 1 \leq j \leq n}}(a_{ij}, i, j),   B=\sum\limits_{\substack{ 1 \leq p \leq m \\ 1 \leq q \leq n}}(b_{pq}, p, q) \in \mathcal{A}$, 
	we have
	$$
	\begin{aligned}
		\phi(A, B)&=\phi(\sum\limits_{\substack{ 1 \leq i \leq m \\ 1 \leq j \leq n}}(a_{ij}, i, j), \sum\limits_{\substack{ 1 \leq p \leq m \\ 1 \leq q \leq n}}(b_{pq}, p, q))=\sum\limits_{\substack{ 1 \leq i \leq m \\ 1 \leq j \leq n}}\sum\limits_{\substack{ 1 \leq p \leq m \\ 1 \leq q \leq n}}\phi((a_{ij}, i, j), (b_{pq}, p, q))\\
		&=\sum\limits_{\substack{ 1 \leq i \leq m \\ 1 \leq j \leq n}}\sum\limits_{\substack{ 1 \leq p \leq m \\ 1 \leq q \leq n}}a_{ij}b_{pq}\lambda^{iq}_{jp}=\sum\limits_{\substack{ 1 \leq i \leq m \\  1 \leq q \leq n}}\sum\limits_{ 1 \leq k \leq r}a_{ik}b_{kq}\lambda^{iq}\\
		&=\tau(\sum\limits_{\substack{ 1 \leq i \leq m \\  1 \leq q \leq n}}(\sum\limits_{ 1 \leq k \leq r}a_{ik}b_{kq}, i, q))=\tau(AB).
	\end{aligned}
	$$
	It follows that $\mathcal{A}$ is zero product determined.	
\end{proof}

Every associative algebra $\mathcal{A}$ becomes a \textit{Jordan algebra}, denoted $\mathcal{A}^+$, if we replace the
original product by the Jordan product. 
An associative algebra $\mathcal{A}$ is said to be \textit{zero Jordan product determined}, if the Jordan
algebra $\mathcal{A}^+$ is zero product determined. 
We  use \textit{zJpd} as an abbreviation for zero Jordan product determined.

In \cite{MR3325219},  An,  Li, and  He 
proved the following theorem.

\begin{theorem}\cite[Theorem 2.1]{MR3325219}\label{th4.22}
	Let $\mathcal{A}$ be an (associative) unital algebra over a field $\mathbb{F}$ of characteristic
	not 2. If $\mathcal{A}$ is generated by idempotents, then $\mathcal{A}$ is zJpd.
\end{theorem}

\begin{theorem}	
	Let $\mathbb{F}$ be a field with $\operatorname{char}\mathbb{F} \neq 2, 3$ and $P$ be an $n \times m$ matrix over $\mathbb{F}$ with $\operatorname{rank} P= r$.  
	If $2 \leq  m, n < \infty$ and $r \geq 2$,  then $\mathcal{A}=\mathfrak{M}(\mathbb{F}, m, n, P)$ is  zero Jordan product determined. 
\end{theorem}

\begin{proof}
	Let  $\phi : \mathcal{A} \times \mathcal{A} \longrightarrow \mathbb{F}$ be a bilinear functional satisfying 
	\begin{equation} \label{eqzjpd}
		X \circ Y=0 \Longrightarrow \phi(X, Y)=0.
	\end{equation}
	For any $1\leq i, p \leq m$ and  $1 \leq j, q \leq n$, 
	denote 
	$	\phi((1, i, j), (1, p, q))=\lambda^{iq}_{jp}$. 
	
	If ($r< j \leq n$ or $r< p \leq m$) and ($r< q \leq n$ or $r< i \leq m$), then $(1, i, j) \circ (1, p, q)=0$ and 
	\begin{equation}\label{eq2.12}
		\lambda^{iq}_{jp}=0.
	\end{equation}

	For any  $1 \leq i \leq m$ and  $1\leq q \leq n$, let 
	$$C=\sum_{1 \leq j \leq n}(c_{ij}, i, j)\quad \text{and}\quad D=\sum_{1 \leq q\leq m}(d_{pq}, p, q),
	$$
	where $c_{ij}$ and $d_{pq}$ are arbitrary elements in $\mathbb{F}$. 
	
	When $r < i \leq m$ or $r < q \leq n$,  we have
	$$CD=(\sum_{1 \leq k \leq r}c_{ik}d_{kq}, i, q) \quad \text{and} \quad C \circ D=(\sum_{1 \leq k \leq r}c_{ik}d_{kq}, i, q).
	$$
	Using \eqref{eq2.12},  we can obtain that
	$$		\phi(C, D)= \sum\nolimits_{\substack{ 1 \leq j \leq n\\ 1 \leq p \leq m}}\phi((c_{ij}, i, j), (d_{pq}, p, q))
	=\sum\nolimits_{\substack{ 1 \leq j \leq r\\ 1 \leq p \leq r}}c_{ij}d_{pq}\lambda^{iq}_{jp}.	
	$$
	Since $C \circ D=0$ implies $\phi(C, D)=0$, it follows from Lemma \ref{danwei}  that 
	\begin{equation}\label{iq}
		(\lambda_{jp}^{iq})_{1 \leq j, p \leq r}=\lambda^{iq} I_r
	\end{equation}
	for some  $\lambda^{iq} \in \mathbb{F}$.
	
	Similarly, when	 $r < p \leq m$ or $r < j \leq n$, we can obtain that  \begin{equation}\label{jp}
		(\lambda_{jp}^{iq})_{1 \leq i, q \leq r}=\lambda_{jp} I_r
	\end{equation} for some  $\lambda_{jp} \in \mathbb{F}$.

	Let $\mathcal{B}$ be a subalgebra of $\mathcal{A}$ consisting of all elements of the form \\ $A=\sum_{\substack{ 1 \leq i \leq r, 1 \leq j \leq r}}(a_{ij}, i, j)$, where 	$a_{ij}\in \mathbb{F}$.
	Then the projection $\eta : \mathcal{A} \longrightarrow \mathcal{B}$  given by $$\sum\limits_{\substack{ 1 \leq i \leq m \\ 1 \leq j \leq n}}(a_{ij}, i, j) \longmapsto \sum\limits_{\substack{ 1 \leq i \leq r \\ 1 \leq j \leq r}}(a_{ij}, i, j)$$
	is an algebraic epimorphism.
	And the injection $\iota : \mathcal{B} \longrightarrow \mathcal{A}$  given by 
	$X \longmapsto X$
	is an algebraic monomorphism.
	
	It is easy  to see that $\mathcal{B}$ is  algebraically isomorphic to $\mathrm{M}_r(\mathbb{F})$. It follows from Theorem \ref{th4.22} that $\mathcal{B}$ is zJpd. 	
	Define a bilinear mapping $\phi_1: \mathcal{B} \times \mathcal{B} \rightarrow \mathbb{F}$ by
	$$
	\phi_1(X, Y) := \phi(\iota(X), \iota(Y))
	$$
	for any $X, Y \in \mathcal{B}$. 
	Since $\iota$ is a homomorphism, $\phi_1$ satisfies the  condition \eqref{eqzjpd}.  Therefore, there exists a linear mapping $\tau_1: \mathcal{B} \longrightarrow \mathbb{F}$ such that  $\phi_1(X, Y) = \tau_1(X \circ Y)$ for any $X, Y \in \mathcal{B}$.	
	
	Now we define a linear functional $\tau :  \mathcal{A} \longrightarrow \mathbb{F}$ by 
	$$\tau(\sum_{\substack{ 1 \leq i \leq m \\ 1 \leq j \leq n}}(a_{ij}, i, j) ) :=\tau_1( \sum_{\substack{ 1 \leq i \leq r \\ 1 \leq j \leq r}}(a_{ij}, i, j))+\sum_{\substack{ r  < i \leq m \\ \text{or~} r < j \leq n}}a_{ij}\lambda^{ij}.
	$$
	
	For any $A=\sum\nolimits_{\substack{ 1 \leq i \leq m, 1 \leq j \leq n}}(a_{ij}, i, j)$ and  $B=\sum\nolimits_{\substack{ 1 \leq p \leq m, 1 \leq q \leq n}}(b_{pq}, p, q)$ in $\mathcal{A}$, we have
	$$
	\begin{aligned}
		A \circ B
		=&(\sum_{\substack{ 1 \leq i \leq r \\ 1 \leq j \leq r}}(a_{ij}, i, j)+\sum_{\substack{ r  < i \leq m \\ \text{or~} r < j \leq n}}(a_{ij}, i, j)) \circ (\sum_{\substack{ 1 \leq p \leq r \\ 1 \leq q \leq r}}(b_{pq}, p, q)+\sum_{\substack{ r  < p \leq m \\ \text{or~} r < q \leq n}}(b_{pq}, p, q))\\
		=&\sum_{\substack{ 1 \leq i \leq r \\ 1 \leq j \leq r}}(a_{ij}, i, j) \circ \sum_{\substack{ 1 \leq p \leq r \\ 1 \leq q \leq r}}(b_{pq}, p, q)+\sum_{\substack{r  < i \leq m\\
				\text{or~} r < j \leq n\\
				\text{or~} r  < p \leq m\\
				\text{or~} r < q \leq n}}(a_{ij}, i, j) \circ (b_{pq}, p, q)\\
		=&\sum_{\substack{ 1 \leq i \leq r \\ 1 \leq j \leq r}}(a_{ij}, i, j) \circ \sum_{\substack{ 1 \leq p \leq r \\ 1 \leq q \leq r}}(b_{pq}, p, q)+\sum_{\substack{ r  < i \leq m \\ \text{or~} r < q \leq n}}(\sum_{1 \leq k \leq r}a_{ik}b_{kq}, i, q)\\
		&+\sum_{\substack{ r  < p \leq m\\ \text{or~} r < j \leq n}}(\sum_{1 \leq k \leq r}b_{pk}a_{kj}, p, j).
	\end{aligned}
	$$
	Thus we can obtain that
	\begin{equation}\label{eq2.38}
		\begin{aligned}
			\tau(A \circ B)=&\tau_1(\sum_{\substack{ 1 \leq i \leq r \\ 1 \leq j \leq r}}(a_{ij}, i, j) \circ \sum_{\substack{ 1 \leq p \leq r \\ 1 \leq q \leq r}}(b_{pq}, p, q))+\sum_{\substack{ r  < i \leq m \\ \text{or~} r < q \leq n}}\sum_{1 \leq k \leq r}a_{ik}b_{kq}\lambda^{iq}\\
			&+\sum_{\substack{ r  < p \leq m\\ \text{or~} r < j \leq n}}\sum_{1 \leq k \leq r}b_{pk}a_{kj}\lambda^{pj},
		\end{aligned}
	\end{equation}		
	and	
	\begin{equation}\label{eq2.39}
		\begin{aligned}
			&\phi(A, B)\\
			=&\phi(\sum_{\substack{ 1 \leq i \leq r \\ 1 \leq j \leq r}}(a_{ij}, i, j)+\sum_{\substack{ r  < i \leq m \\ \text{or~} r < j \leq n}}(a_{ij}, i, j), \sum_{\substack{ 1 \leq p \leq r \\ 1 \leq q \leq r}}(b_{pq}, p, q)+\sum_{\substack{ r  < p \leq m \\ \text{or~} r < q \leq n}}(b_{pq}, p, q))\\
			=&\phi_1(\eta(\sum_{\substack{ 1 \leq i \leq r \\ 1 \leq j \leq r}}(a_{ij}, i, j)), \eta(\sum_{\substack{ 1 \leq p \leq r \\ 1 \leq q \leq r}}(b_{pq}, p, q)) )+\sum_{\substack{r  < i \leq m\\
					\text{or~} r < j \leq n\\
					\text{or~} r  < p \leq m\\
					\text{or~} r < q \leq n}}\phi((a_{ij}, i, j), (b_{pq}, p, q))\\
			=&\tau_1(\sum_{\substack{ 1 \leq i \leq r \\ 1 \leq j \leq r}}(a_{ij}, i, j) \circ \sum_{\substack{ 1 \leq p \leq r \\ 1 \leq q \leq r}}(b_{pq}, p, q))
			+\sum_{\substack{r  < i \leq m\\
					\text{or~} r < j \leq n\\
					\text{or~} r  < p \leq m\\
					\text{or~} r < q \leq n}}a_{ij}b_{pq}\lambda^{iq}_{jp}.
		\end{aligned}
	\end{equation}

	It follows from \eqref{eq2.12} and \eqref{iq} that 
	\begin{equation}\label{eq2.40}
		\begin{aligned}	
			\sum_{\substack{r  < i \leq m\\
					\text{or~} r < j \leq n\\
					\text{or~} r  < p \leq m\\
					\text{or~} r < q \leq n}}a_{ij}b_{pq}\lambda^{iq}_{jp}=&	\sum_{\substack{ r < j \leq n\\
					\text{or~} r  < p \leq m}}\sum_{\substack{1  \leq i \leq r\\
					1 \leq q \leq r}} a_{ij}b_{pq}\lambda^{iq}_{jp}+
			\sum_{\substack{ r < q \leq n\\
					\text{or~} r  < i \leq m}}\sum_{\substack{1  \leq j \leq r\\
					1 \leq p \leq r}} a_{ij}b_{pq}\lambda^{iq}_{jp}\\
			=&\sum_{\substack{ r < j \leq n\\
					\text{or~} r  < p \leq m}}\sum_{1\leq k \leq r} a_{kj}b_{pk}\lambda_{jp}
			+
			\sum_{\substack{ r < q \leq n\\
					\text{or~} r  < i \leq m}}\sum_{1 \leq k \leq r} a_{ik}b_{kq}\lambda^{iq}.
		\end{aligned}
	\end{equation}
	
	Combing \eqref{eq2.38}, \eqref{eq2.39} and \eqref{eq2.40}, it is sufficient for us to prove that 
	$ \lambda^{ij}=\lambda_{ji}
	$
	when $r  < i \leq m$ or $r < j \leq n$.
	
	For any  $1 \leq i \neq k \leq  r$ and $r < j \leq n$, let 
	\begin{eqnarray*}	
		&E=(1, k, i)+(1, i, k)+(1, k, j), \quad &F=(-1, k, i)+(1, i, k)+(-1, k, j),\\
		&E^\prime=(1, k, i)+(1, i, k), \quad\quad\quad\quad\quad &F^\prime=(-1, k, i)+(1, i, k).
	\end{eqnarray*}

	Thus  $E \circ F=0$ and $E^\prime \circ F^\prime=0$. Then by \eqref{eqzjpd}, we have $\phi(E, F)=\phi(E^\prime, F^\prime)=0$.
	It follows that 
	\begin{equation}\label{eq2.42}
		-\lambda^{ki}_{jk}+\lambda^{kk}_{ji}-\lambda^{kj}_{jk}-\lambda^{kj}_{ik}-\lambda^{ij}_{kk}=0.
	\end{equation}
	By \eqref{iq},  \eqref{jp} and \eqref{eq2.42}, we can obtain $ \lambda^{ij}=\lambda_{ji}$ when $1 \leq i  \leq  r$ and $r < j \leq n$.
	
	Similarly, we have $ \lambda^{ij}=\lambda_{ji}$ when $1 \leq j  \leq  r$ and $r < i \leq n$.
	
	For any  $r < i \leq  m$, $r< j \leq n$ and $1 \leq k \leq r$, let 
	$$
	G=(1, i, k)+(-1, k, j) \quad \text{and} \quad
	H=(1, i, k)+(1, k, j).
	$$
	Thus  $G \circ H=0$ and $\phi(G, H)=0$.
	It follows that 
	\begin{equation}\label{eq2.45}
		\lambda^{ik}_{ki}-\lambda^{kk}_{ji}+\lambda^{ij}_{kk}-\lambda^{kj}_{jk}=0.
	\end{equation}
	By \eqref{iq},  \eqref{jp} and \eqref{eq2.45}, we can obtain $ \lambda^{ij}=\lambda_{ji}$ when $r < i \leq m$ and $r < j \leq n$. 
	In  conclusion, $ \lambda^{ij}=\lambda_{ji}
	$
	when $r  < i \leq m$ or $r < j \leq n$. Thus we have $\phi(A, B)=\tau(A \circ B)$ for all $A, B \in \mathcal{A}$, which means that $\mathcal{A}$ is zJpd.
\end{proof}

\begin{remark}	
	Let $R$ be a unital Banach algebra, $I$ and $\Lambda$ be arbitrary index sets, and the sandwich matrix $P$ be a non zero $\Lambda \times I$
	matrix over $R$ such that $\|P\|_{\infty}=\sup \left\{\left\|p_{m i}\right\|: m \in \Lambda, i \in I\right\} \leq 1$. 
	The vector space $\mathfrak{M}(R, P, I, \Lambda)$ of all $I \times \Lambda$ matrices over $R$  with $\ell_1$-norm $\|(a_{i m})\|=\sum_{i \in I, m \in \Lambda}\|a_{i m}\|<\infty$ and product $A \bullet B=A P B$
	is a Banach algebra that called the \textit{$\ell_1$-Munn Banach algebra}, see \cite{MR1674619, MR2650729}.
\end{remark}

Using \cite[Theorem 4.1]{MR4215991} and \cite[Theorem 5.14]{MR4311188}, it is easy to prove the following theorem.

\begin{theorem}
	Let $\mathbb{F}$ be a field of real or complex numbers, $I$ and $\Lambda$ be arbitrary index sets,  $P$ be a non zero $\Lambda \times I$
	matrix over $\mathbb{F}$ such that $\|P\|_{\infty}=\sup \left\{\left\|p_{m i}\right\|: m \in \Lambda, i \in I\right\} \leq 1$. If $I$ or $\Lambda$ is finite and $P$ is invertible over 
	the group of all invertible elements of $R$, then $\mathfrak{M}(R, P, I, \Lambda)$ is a  zpd Banach algebra.
\end{theorem}

\section*{Acknowledgements}
This research was partly supported by the National Natural Science Foundation of China (Grant No.
11871021).

\providecommand{\bysame}{\leavevmode\hbox to3em{\hrulefill}\thinspace}

\end{document}